\documentclass[10pt]{article}
\usepackage{amsmath,amssymb,amsfonts,amscd,amsthm,mathtools} 
\usepackage[subnum]{cases}
\usepackage{graphicx}
\usepackage{color}                    
\usepackage{enumerate}
\usepackage{authblk}

\newcommand{\comment}[1]{\emph{\color{red}}}
\usepackage[parfill]{parskip}
\newcommand{\eps}{\varepsilon}

\newcommand\R{\mathbb{ R}}

\newcommand\E{\mathbb{E}}
\newtheorem{theorem}{Theorem}[section]

\newtheorem{corollary}[theorem]{Corollary}
\newtheorem{lemma}[theorem]{Lemma}
\newtheorem{remark}[theorem]{Remark}

\title{Discrete and continuum links to a nonlinear coupled transport problem of interacting populations}
\author[1]{M. H. Duong\footnote{Corresponding author, email: m.h.duong@warwick.ac.uk }}
\author[2]{A. Muntean}
\author[2]{O. M. Richardson}
\affil[1]{University of Warwick, UK. }
\affil[2]{Karlstad University, Sweden.}
\begin{document}
\maketitle
\begin{abstract}
We are interested in exploring interacting particle systems that can be seen as microscopic models for a particular structure of  coupled transport flux
arising when  different populations are jointly evolving. The scenarios we have in mind are inspired by the dynamics of pedestrian flows in open spaces and are intimately connected to cross-diffusion and thermo-diffusion problems holding a variational structure.  The tools we use include a suitable structure of the relative entropy controlling TV-norms, the construction of Lyapunov functionals and particular closed-form solutions to nonlinear transport equations, a hydrodynamics limiting procedure due to Philipowski, as well as the construction of numerical approximates to both the continuum limit problem in 2D and to the original interacting particle systems.  

\end{abstract}
\section{Introduction}
\label{intro}
The starting point of the results presented in this paper is the following question\footnote{This question was posed by Prof. M. Mimura (Meiji, Tokyo, Japan)  to A. Muntean during a visit at Meiji University}:\\
{\em Can one design a system of interacting particles that converges in some suitable limit to the following system of nonlinearly coupled system of transport equations:
\begin{numcases}{}
\partial_t u=(u(u+v)_x)_x,\label{eq: u}\\
\partial_t v=(v(u+v)_x)_x,\label{eq: v}
\end{numcases}
with initial conditions $u(0,x)=u_0(x)$ and  $v(0,x)=v_0(x) \ (x\in \mathbb{R})$}?
Here $u$ and $v$ refer to mass concentrations of some chemical species which are participating in a non-competitive manner in a joint transport process.  The background of the question (and interest of M. Mimura) is connected to the role pheromones play in influencing the aggregation phenomenon, one of the main survival mechanisms in insects, birds and animal colonies; we refer the reader to \cite{Funaki} for more on this context. It is worth noting the coupled structure of the transport fluxes resembles situations arising in cross-diffusion and thermo-diffusion. Compare  \cite{Mazur} for the thermodynamical foundations of cross- and thermo- diffusion and \cite{Vanag} for a nice paper illustrating the role of cross-diffusion mechanisms towards pattern formation in chemical systems.  Our own interest in this framework targets at the fundamental  understanding of well-observed  optimal self-organization behaviours (e.g  lane formation in counter-flows) exhibited by the motion of  pedestrian flows (cf. e.g. \cite{Helbing_self} and references cited therein).

Interestingly,  due to the symmetry in the structure of the equations, the system (\ref{eq: u}) -- (\ref{eq: v}) admits a direct interpretation from the porous media theory point of view, which later turns out to be very useful in understanding mathematically the particle system origin of this transport problem.

We assume that $u$ and $v$ denote two populations (of pedestrians, ants, chemical species, etc.) that like to travel together. Think, for instance,  of  a pair of large families of individuals that wish to reach perhaps a common destination or target, under the basic assumption that besides some kind of social pairwise repulsion and adherence to the same drift there are no other interactions  in the crowd made of the two populations.
This basic situation can be modelled as a system of continuity equations 
 \begin{eqnarray*}
 \partial_t u+{\rm div}(uw)&=&0,\\
  \partial_t v+{\rm div}(vw)&=&0,
 \end{eqnarray*}
 where $w$ is the common drift to which the two populations adhere. The velocity vector $w$ is assumed now to comply with Darcy's law
 \begin{equation}\label{Darcy}
 w=- \frac{K}{\mu}\nabla p.
 \end{equation}
 In (\ref{Darcy}),  $\frac{K}{\mu}\ \in (0,\infty)$ denotes the permeability coefficient (usually a tensor for a heterogeneous region) and $p$ is the  total (social) pressure in the system. 
Now, making  the ansatz on the structure of the pressure  
\begin{equation*}
p=\mu(u+v), 
\end{equation*}
and then summing up the above continuity equations, we obtain the 
system \eqref{eq: u}-\eqref{eq: v}, where for simplicity we take $K\equiv 1$.

The paper is organised as follows. In Section \ref{sec: analytical}, we provide some basic analytic understanding of \eqref{eq: u}-\eqref{eq: v} by transforming the system to an equivalent one, showing the local well-posedness, constructing a special class of solutions and proving the preservation of relative entropy and the consequences this has on the large-time behaviour of the system. In Section \ref{sec: particle}, we introduce a stochastically interacting many-particle system to approximate \eqref{eq: u}-\eqref{eq: v}. Finally, Section \ref{sec: numerical} presents numerical illustrations of the particle system, indicating numerical evidence on the expected convergence.
\section{Analytical results}
\label{sec: analytical}
In this section, we provide a couple of analytical results on the continuum model. We first transform the system \eqref{eq: u}-\eqref{eq: v} to an equivalent one. Using this transformation, we ensure in a straightforward way the  local existence of classical solutions. In addition, we construct a special class of solutions and show remarkable properties of these solutions, especially concerning the preservation of the relative entropy.
\subsection{An equivalent system}
Defining $w:=u+v$, we see that $w$ solves the following porous media-like equation:
\begin{equation}
\partial_t w =\frac{1}{2}\partial_{xx} (w^2),\quad w(0,x)=u_0(x)+v_0(x).
\end{equation}
We transform the system \eqref{eq: u}-\eqref{eq: v} posed for $(u,v)$ into the following system for $(w,u)$:
\begin{numcases}{}
\partial_t w=\frac{1}{2}\partial_{xx}(w^2),\quad  w(0,x)=w_0(x),\label{eq: eqn for w}\\
\partial_t u=\partial_x(u w_x),\quad u(0,x)=u_0(x).\label{eq:eqn for u}
\end{numcases}
Conversely, suppose that $(w,u)$ satisfies the system \eqref{eq: eqn for w}-\eqref{eq:eqn for u}. Then $(u, v)$, where $v=w-u$, satisfies the original system \eqref{eq: u}-\eqref{eq: v}. Therefore, the two systems are equivalent.

The transformation has two advantages. First, the new system \eqref{eq: eqn for w}-\eqref{eq:eqn for u} is only one-sided coupled in the sense that one can solve \eqref{eq: eqn for w} independently to obtain $w$, and then substitute to find $u$ from \eqref{eq:eqn for u} with $w$ given. Second, \eqref{eq: eqn for w} is the famous Boussinesq's equation of groundwater flow, while \eqref{eq:eqn for u} is the standard continuity equation. Both equations have been studied extensively and have a rich literature. Therefore, we can apply existing methods and techniques to handle them from the mathematical analysis point of view.
\subsection{A general solution to the continuity equation by the method of characteristics}
Let $V$ be a given velocity field and $f_0:\R\to \R$ be a given function. We first seek solutions for the following general continuity equation
\begin{equation}
\label{eq: continuity eqn}
\frac{\partial f}{\partial t}+\frac{\partial}{\partial x}[V(x,t)f]=0 \mbox{ with } f(x,0)=f_0(x).
\end{equation}
We consider the following ordinary differential equation (ODE):
\begin{equation}
\label{eq: ODE}
\frac{d}{dt}X(t)=V(X(t),t),\quad X(0)=x.
\end{equation}
The solution of this ODE is $X(t)=F(x,t)$. Conversely, we also can regard $x$ as a function of $X(t)$, i.e., $x=G(X(t),t)$, where $G: \R\times\R\ni(y,t)\mapsto G(y,t)\in \R$ and $G(y,0)=y$. 
\begin{lemma}[Solving the continuity equation, see e.g.\cite{Clement1978}] The function
\label{lem: sol of continuity}
\begin{equation}
\label{eq: sols of continuity}
f(x,t)=f_0(G(x,t))\frac{\partial G}{\partial x}(x,t)=-\frac{f_0(G(x,t))}{V(x,t)}\frac{\partial G}{\partial t}(x,t)
\end{equation}
solves the continuity equation \eqref{eq: continuity eqn}.
\end{lemma}
\subsection{Classical solutions}
The first result of this paper refers to the local existence of classical solutions of \eqref{eq: eqn for w}-\eqref{eq:eqn for u}. Let $T>0$ be sufficiently large but fixed and let $(w_0, u_0)$ be given. We say that the couple $(w,u)$, where $w,u:[0,T]\times\R\mapsto \R$ is a classical solution to the system \eqref{eq: eqn for w}-\eqref{eq:eqn for u} if  $w,u \in C^{2,1}([0,T],\R)$ and satisfy \eqref{eq: eqn for w}-\eqref{eq:eqn for u}.
\begin{theorem}
\label{theo: local existence w}
 Suppose that $w_0$ and $u_0$ are continuous functions in $\R$ with 
\begin{equation*}
\varepsilon\leq w_0(x)\leq \frac{1}{\varepsilon},
\end{equation*}
for some $\varepsilon>0$ and all $x\in \R$. There exists $T^*\in (0,T)$ such that the system \eqref{eq: eqn for w}-\eqref{eq:eqn for u} has a classical solution in $C^{2,1}([0,T^*],\R)$
\end{theorem}
\begin{proof}
This theorem is a direct consequence of \cite[Theorem 3.1]{Vaz06} for
the (global) existence of the Boussinesq's solution and of the Peano's theorem for the local existence of the characteristic trajectory.
\end{proof}
\subsection{A special class of solutions}
Due to the particular structure of the system \eqref{eq: eqn for w}-\eqref{eq:eqn for u}, namely \eqref{eq: eqn for w} being the Boussinesq's equation and \eqref{eq:eqn for u} being the continuity equation, we are able to construct a special class of solutions. We consider a solution profile of quadratic functions for $w$ and then find $u$ accordingly. The idea of the former has been used before, see for instance~\cite{King1993}.
 
\textit{Step 1.} We rewrite \eqref{eq: eqn for w} as
\begin{equation}
\label{eq: rewrite eqn for w}
\partial_t w=(\partial_x w)^2+w\partial_{xx} w.
\end{equation}
We consider solutions to \eqref{eq: rewrite eqn for w} of the form
\begin{equation}
w(t,x)=A(t)-B(t)x^2 \mbox{ where } w(0,x)=w_0(x).
\end{equation}
By substituting this form in \eqref{eq: rewrite eqn for w}, we obtain the following system
\begin{align*}
\frac{dA(t)}{dt}&=-2A(t)B(t),\quad A(0)=w_0(0)=:a,
\\\frac{dB(t)}{dt}&=-6B(t)^2, \quad B(0)=w_0(1)-w_0(0)=:b,
\end{align*}
which finally leads to 
\begin{equation}
A(t)=a(6bt+1)^{-\frac{1}{3}},\quad B(t)=\frac{b}{6bt+1}.
\end{equation}
Therefore, 
\begin{equation}
\label{eq: special w}
w(t,x)=a(6bt+1)^{-\frac{1}{3}}-\frac{b}{6bt+1}x^2.
\end{equation}
\textit{Step 2.} Substituting \eqref{eq: special w} back into \eqref{eq:eqn for u}, we obtain the following continuity equations in terms of $u$.
\begin{equation*}
\partial_t u+\partial_x\left[\frac{2bx}{6bt+1}u\right]=0.
\end{equation*}
We now apply Lemma \ref{lem: sol of continuity} to solve this equation. The ODE \eqref{eq: ODE} becomes
\begin{equation*}
\frac{d}{dt}X(t)=\frac{2b X(t)}{6bt+1},\quad X(0)=x,
\end{equation*}
which gives
\begin{equation*}
X(t)=x(6bt+1)^\frac{1}{3}. \quad\text{ Hence, we get }\quad G(y,t)=\frac{y}{(6bt+1)^\frac{1}{3}}.
\end{equation*}
Therefore, we obtain
\begin{equation*}
u(x,t)=u_0(G(x,t))\frac{\partial G}{\partial x}(x,t)=u_0\left(\frac{x}{(6bt+1)^\frac{1}{3}}\right)\frac{1}{(6bt+1)^\frac{1}{3}}.
\end{equation*}
Concluding, we have obtained a special solution to the system \eqref{eq: eqn for w}-\eqref{eq:eqn for u} as follows
\begin{equation*}
w(t,x)=a(6bt+1)^{-\frac{1}{3}}-\frac{b}{6bt+1}x^2, \quad\text{and}\quad u(x,t)=u_0\left(\frac{x}{(6bt+1)^\frac{1}{3}}\right)\frac{1}{(6bt+1)^\frac{1}{3}},
\end{equation*}
for some $a,b\in\R$. If one considers non-negative solutions, then one should take the positive parts of these expressions.

\subsection{Preservation of the relative entropy and consequences}

We observe that our original system is symmetric in the sense that if we swap $u$ and $v$ in \eqref{eq: u}-\eqref{eq: v} then the system remains unchanged. Therefore, if the initial data $u_0$ and $v_0$ are equal, then it is expected that $u$ and $v$ will be equal at any later time, which is a necessary condition for uniqueness. Two mathematical questions naturally arise at this point:
\begin{enumerate}[(i)]
\item How to prove equality of $u$ and $v$ rigorously?
\item If $u_0$ and $v_0$ are not equal, can we still quantify the distance between $u(t)$ and $v(t)$ in terms of the initial data?
\end{enumerate}
In this section, we provide affirmative answers to these questions using the concept of  relative entropy and the total variation metric. We generalise the results of this section (and of the previous one) to a more general system in Section \ref{sec: generalisation}. It will become clear that structure of the system matches nicely with the concept of the relative entropy.

We now recall the definition of the relative entropy, the total variation metric, as well as a relationship between the twos. We refer the reader to the survey paper \cite{GozlanLeonard2010} for more information.

Let $f(x)\,dx$ and $g(x)\,dx$ be two probability densities on $\mathbb{R}$. The relative entropy of $f$ with respect to $g$ is defined by
\begin{equation*}
    H(f||g):=\int_{\R} \frac{f(x)}{g(x)}\log\frac{f(x)}{g(x)}g(x)\,dx.
\end{equation*}
The total variation distance between $f(x)\,dx$ and $g(x)\,dx$ is defined as
\begin{equation*}
    TV(f,g):=||f-g||_{L^1}=\int_\mathbb{R} \Big|\frac{f(x)}{g(x)}-1\Big|g(x)\,dx.
\end{equation*}
Note that the relative entropy is always non-negative and it is equal to $0$ if and only if $f=g$. Although it is not a distance (it satisfies neither the triangle inequality nor the symmetry condition), it is a useful quantity to measure the difference between two probability measures and has been used extensively in the literature. In addition, it also provides an upper-bound for the total variation distance $TV(f,g)$ by Pinsker's inequality, see for instance \cite[Theorem 1.1]{GozlanLeonard2010}, as 
\begin{equation}
\label{eq: Pinsker inequality}
TV(f,g)\leq \sqrt{2 H(f||g)}.
\end{equation}

\begin{theorem}
\label{theo: rel entropy} Suppose that $u,v$ are classical solutions to the system \eqref{eq: u}-\eqref{eq: v} that decay sufficiently fast at infinity. Then the function $t\mapsto H(u(t)||v(t))$ is constant, i.e., for any $0<t<T^*$ it holds
\begin{equation}
H(u(t)||v(t))=H(u_0||v_0).
\end{equation}
\end{theorem}
\begin{proof}
We calculate the time-derivative of $t\mapsto H(u(t)||v(t))$ as follows (the time variable $t$ is dropped in the right-hand side for simplicity of notation)
\begin{align}
\frac{d}{dt}H(u(t)||v(t))&=\frac{d}{dt}\int u\log\frac{u}{v}\,dx\nonumber
\\&=\int\Big[\big(\log\frac{u}{v}\big)\partial_t u+u\frac{(v\partial_tu -u\partial_t v)/v^2}{u/v}\Big]\,dx\notag
\\&=\int\Big[\big(\log\frac{u}{v}\big)\partial_t u-\frac{u}{v}\partial_t v\Big]\,dx\label{eq:prob_meas}
\\&=\int\Big[\big(\log\frac{u}{v}\big)[u(u+v)_x]_x-\big(\frac{u}{v}\big)[v(u+v)_x]_x\Big]\,dx\label{eq: integration}
\\&=\int (u+v)_x\Big[-u\frac{\partial_x(u/v)}{u/v}+v\partial_x(u/v)\Big]\,dx=0.\notag
\end{align}
Note that \eqref{eq:prob_meas} follows due to $\int\frac{\partial}{\partial t}u = 0$ which is a consequence of conservation of mass. In \eqref{eq: integration} we have used integration by parts where the boundary terms vanish due to the assumption on the decay of the solution.
\end{proof}
\begin{corollary} 
\label{cor: TV estimate}
For any $0<t<T^*$, it holds that
\begin{equation*}
TV(u(t)||v(t))\leq \sqrt{2H(u_0||v_0)}.
\end{equation*}
\end{corollary}
\begin{proof}
This inequality is direct consequence of the Pinsker's inequality \eqref{eq: Pinsker inequality} and Theorem \ref{theo: rel entropy}.
\end{proof}
\begin{corollary} Suppose that $u_0=v_0$, then $(u,v)=(\frac{1}{2}w,\frac{1}{2}w)$, where $w$ solves \eqref{eq: eqn for w}, is the unique solution to the system \eqref{eq: u}-\eqref{eq: v}.
\end{corollary}
Heuristically, note that if $u_0=v_0$ and $(u,v)$ satisfies the system \eqref{eq: u}-\eqref{eq: v}, then so does $(v,u)$. To guarantee the uniqueness, it follows that $u=v=\frac{1}{2}w$. Theorem \ref{theo: rel entropy} offers a much stronger result.
\begin{proof}
This is a direct consequence of Theorem \ref{theo: rel entropy} (or Corollary \ref{cor: TV estimate}). Since $u_0=v_0$, we have $H(u_0||v_0)=0$. Then it follows from Theorem \ref{theo: rel entropy} that $H(u(t)||v(t))=0$ for all $t>0$, which in turn implies that $u(t)=v(t)=\frac{1}{2}w(t)$ for all $t>0$. 
\end{proof}
\subsection{Generalisations}
\label{sec: generalisation}
 It is worth noting that Theorem \ref{theo: local existence w} and Theorem \ref{theo: rel entropy} can be extended to a more general system of the form
\begin{numcases}{}
\partial_t u=\Big[u\big[f(u+v)\big]_x\Big]_x,\quad u(0)=u_0(x),\label{eq: general 1}\\
\partial_t v=\Big[v\big[f(u+v)\big]_x\Big]_x,\quad v(0)=v_0(x)\label{eq: general 2}.
\end{numcases}
The transformed system for $(w,u)$, where $w=(u+v)$, now becomes
\begin{numcases}{}
\partial_t w=\partial_{x}(w\partial_x f(w)),\quad w(0)=w_0(x),\nonumber\\
\partial_t u=\partial_x(u \partial_x f(w)),\quad u(0)=u_0(x).\nonumber
\end{numcases}
For instance, if $f(z)=z^{m-1}$ for some $m>1$, then the equation for $w$ becomes
\begin{equation*}
\partial_t w=\frac{m-1}{m}\partial_{xx}(w^m).
\end{equation*}
This is the standard porous medium equation. Hence,  we can apply \cite[Theorem 3.1]{Vaz06} again; thus Theorem \ref{theo: local existence w} still holds true. We now show that Theorem \ref{theo: rel entropy} can also be extended to general shapes of $f$.
\begin{theorem} 
\label{theo: general rel entropy}
Suppose that $u,v$ are classical solutions to the general system \eqref{eq: general 1}-\eqref{eq: general 2} that decay sufficiently fast at infinity. Then the function $t\mapsto H(u(t)||v(t))$ is constant.
\end{theorem}
\begin{proof}
Similar computations as in the proof of Theorem \ref{theo: rel entropy} give
\begin{equation*}
\frac{d}{dt}H(u(t)||v(t))=\int \partial_x f(u+v)\Big(-u\frac{\partial_x(u/v)}{u/v}+v\partial_x(u/v)\Big)\,dx=0.
\end{equation*}
\end{proof}
\begin{remark}
We note that the common relation that makes the relative entropies in both Theorem \ref{theo: rel entropy} and Theorem \ref{theo: general rel entropy} vanish is
\begin{equation*}
-u\frac{\partial_x(u/v)}{u/v}+v\partial_x(u/v)=0.
\end{equation*}
Tracing back this relation in the calculations, this property appears because of the combination of three ingredients: the formula of the relative entropy, the symmetry of the system, and the formulas of the continuity equations. The last two properties together form the structure of the system.
\begin{enumerate}[i)]
\item The continuity equations provide that 
\begin{equation*}
\partial_t u=\partial_x[u X]\quad\text{and}\quad \partial_t v=\partial_x[v Y],
\end{equation*}
\item The symmetry of \eqref{eq: u} and \eqref{eq: v} means that $X=Y$ (so that if we swap $u$ and $v$, the system is unchanged).
\end{enumerate} 
In other words, we find that the relative entropy is constant essentially due to the structure of the system.
\end{remark}

\section{Particle system approach}
\label{sec: particle}
In this section, we introduce a many-particle system that includes coupled weakly interacting stochastic differential equations. We formally show that the empirical measures associated to this system converge to solutions of the original system \eqref{eq: u}-\eqref{eq: v}. The rigorous proof will be given in a separate paper.

We consider the following particle system: 
\begin{align}
    \begin{split}
&dX_t^{i,\varepsilon}=-\frac{1}{n}\sum\limits_{j=1}^n \left(V'_\varepsilon(X^{j,\varepsilon}_t-X^{i,\varepsilon}_t)+V'_\varepsilon(Y^{j,\varepsilon}_t-X^{i,\varepsilon}_t)\right)\,dt+\varepsilon dW^i_t,
\\&dY_t^{i,\varepsilon}=-\frac{1}{n}\sum\limits_{j=1}^n \left(V'_\varepsilon(X^{j,\varepsilon}_t-Y^{i,\varepsilon}_t)+ V'_\varepsilon(Y^{j,\varepsilon}_t-Y^{i,\varepsilon}_t)\right)\,dt+\varepsilon dW^{n+i}_t,
    \end{split}
    \label{eq:particle_system}
\end{align}
for $i=1,\ldots, n$, where $\{W^i\}_{i=1}^{2n}$ are independent standard Wiener processes, $\{V_\varepsilon\}_{\varepsilon\geq 0}$ are a sequence of smooth functions which are chosen later on. Note that the system in \eqref{eq:particle_system} can be seen as a generalisation of the many-particle system arising in \cite{Philipowski2007} to our model of coupled interactions of two species. Remark also that in \cite{DiFrancescoFagioli2013,DiFrancescoFagioli2016}, the authors studied similar systems but in the absence of the stochastic noise.

We define the following empirical measures
\begin{equation}
\label{eq: empirical measure}
u^{n,\varepsilon}_t=\frac{1}{n}\sum\limits_{i=1}^n\delta_{X^{i,\varepsilon}_t},\qquad v^{n,\varepsilon}_t= \frac{1}{n}\sum\limits_{i=1}^n\delta_{Y^{i,\varepsilon}_t}.
\end{equation}

We now formally derive the system \eqref{eq: u}-\eqref{eq: v} in two steps:
\begin{enumerate}[\textbf{Step} 1:]
\item {\em Hydrodynamic limit}, as $n$ tends to infinity: \begin{equation*}
u^{n,\varepsilon}_t\rightharpoonup u^\varepsilon_t,\qquad v^{n,\varepsilon}_t\rightharpoonup v^\varepsilon_t,
\end{equation*}
where $(u^\varepsilon,v^\varepsilon)$ solves a system which depends on $V'_\varepsilon$ and with some viscous terms.
\item {\em Viscosity limit}, as $\varepsilon$ tends to $0$:
\begin{equation*}
u^\varepsilon_t\rightharpoonup u_t,\qquad v^\varepsilon_t\rightharpoonup v_t,
\end{equation*}
where $(u,v)$ solves the original system.
\end{enumerate}

The derivation explains the choice of scalings occurring in the many-particle system. Now, we perform the first step.

\textbf{Step 1 (Hydrodynamic limit):} Let $f$ be a sufficiently smooth function. By definition \eqref{eq: empirical measure} of the empirical measure $u_t^{n,\varepsilon}$, we have
\begin{equation*}
\langle f, u_t^{n,\varepsilon}\rangle := \int f(x) u_t^{n,\varepsilon}(dx)=\frac{1}{n}\sum_{i=1}^n f(X^{i,\varepsilon}_t).
\end{equation*}
Using It\^o's lemma and definition of the empirical measures in \eqref{eq: empirical measure}, we derive that
\begin{align*}
d\langle f, u_t^{n,\varepsilon}\rangle&=\langle -f' V'_\varepsilon\ast (u^{n,\varepsilon}_t+v^{n,\varepsilon}_t)+\frac{\varepsilon^2}{2}f'',u^{n,\varepsilon}_t\rangle\,dt+\varepsilon\sum_{i=1}^n f'(X^{i,\varepsilon}_t)dW^i_t,
\\ d\langle f, v_t^{n,\varepsilon}\rangle&=\langle -f' V'_\varepsilon\ast (u^{n,\varepsilon}_t+v^{n,\varepsilon}_t)+\frac{\varepsilon^2}{2}f'',v^{n,\varepsilon}_t\rangle\,dt+\varepsilon\sum_{i=1}^n f'(Y^{i,\varepsilon}_t)dW^{n+i}_t,
\end{align*}
where $\ast$ denotes a convolution operator. By taking the expectation, the Brownian terms vanish, and we obtain that
\begin{align*}
\partial_t \E\langle f, u_t^{n,\varepsilon}\rangle&=\E \langle \partial_x\big[u^{n,\varepsilon}_t V'_\varepsilon\ast(u^{n,\varepsilon}_t+v^{n,\varepsilon}_t)\big]+\frac{\varepsilon^2}{2}\partial_{xx}u^{n,\varepsilon}_t,f \rangle,
\\\partial_t \E\langle f, v_t^{n,\varepsilon}\rangle&=\E \langle \partial_x\big[v^{n,\varepsilon}_t V'_\varepsilon\ast(u^{n,\varepsilon}_t+v^{n,\varepsilon}_t)\big]+\frac{\varepsilon^2}{2}\partial_{xx}v^{n,\varepsilon}_t,f \rangle.
\end{align*}
The key point is that we now suppose that $u^{n,\varepsilon}_t\xrightharpoonup{\,n\to\infty\,} u^\varepsilon_t, v^{n,\varepsilon}_t\xrightharpoonup{\,n\to\infty\,} v^\varepsilon_t$ where $u^\varepsilon_t$ and $v^\varepsilon_t$ are deterministic profiles. Then the pair $(u^\varepsilon_t, v^\varepsilon_t)$ satisfies for all $f$ the following identities:
\begin{align*}
\partial_t \langle f, u_t^{\varepsilon}\rangle&=\langle \partial_x\big[u^{\varepsilon}_t V'_\varepsilon\ast(u^{\varepsilon}_t+v^{\varepsilon}_t)\big]+\frac{\varepsilon^2}{2}\partial_{xx}u^{\varepsilon}_t,f \rangle,
\\\partial_t \langle f, v_t^{\varepsilon}\rangle&=\langle \partial_x\big[v^{\varepsilon}_t V'_\varepsilon\ast(u^{\varepsilon}_t+v^{\varepsilon}_t)\big]+\frac{\varepsilon^2}{2}\partial_{xx}v^{\varepsilon}_t,f \rangle,
\end{align*}
which are respectively weak formulations of
\begin{align*}
\partial_t u^\varepsilon_t&=\partial_x\big[u^{\varepsilon}_t V'_\varepsilon\ast(u^{\varepsilon}_t+v^{\varepsilon}_t)\big]+\frac{\varepsilon^2}{2}\partial_{xx}u^{\varepsilon}_t,
\\\partial_t v^\varepsilon_t&=\partial_x\big[v^{\varepsilon}_t V'_\varepsilon\ast(u^{\varepsilon}_t+v^{\varepsilon}_t)\big]+\frac{\varepsilon^2}{2}\partial_{xx}v^{\varepsilon}_t.
\end{align*}

\textbf{Step 2 (Viscosity limit):} Assume that $u^\varepsilon_t\rightharpoonup u_t$, $ v^\varepsilon_t\rightharpoonup v_t$, $V_\varepsilon\rightharpoonup \delta$ and such that the diffusive terms vanish in the limit $\varepsilon\rightarrow0$. Then, since
\begin{align*}
V_\varepsilon'\ast (u_t^\varepsilon+v_t^\varepsilon)(x)&=\int V_\varepsilon'(x-y)(u_t^\varepsilon+v_t^\varepsilon)(y)\,dy=\int \partial_y (u_t^\varepsilon+v_t^\varepsilon)(y)V_\varepsilon(x-y)\,dy
\\&\to \int \partial_y (u_t^\varepsilon+v_t^\varepsilon)(y)\delta_{x-y}\,dy=\partial_x (u_t^\varepsilon+v_t^\varepsilon)(x),
\end{align*}
we formally get
\begin{align*}
&\partial_t u=\partial_x[u\partial_x(u+v)],
\\&\partial_t v=\partial_x[v\partial_x(u+v)],
\end{align*}
which is exactly the system \eqref{eq: u}-\eqref{eq: v}. To show rigorously the viscosity limit, we rely on the techniques presented in \cite{Evans}. To keep a concise presentation, we omit to complete  the line of the arguments here and postpone them to a forthcoming paper. 

\section{Numerical simulations}
\label{sec: numerical}
In this section, we illustrate numerically in 2D the solution of \eqref{eq:particle_system} for specific initial data and explore numerically to which extent the continuum model \eqref{eq: u}-\eqref{eq: v} can be approximated based on \eqref{eq:particle_system}.
\subsection{Continuum system}
We naturally extend the one-dimensional model to two dimensions using the following notations: Denote by $T*$ the final observation time. The populations $u=u(x,y,t)$ and $v=v(x,y,t)$, where  $u,v: \Omega\times[0,T^*] \to \mathbb{R}$, with $\Omega\subset \mathbb{R}^2$, satisfy
\begin{align}
    \begin{split}
	\partial_tu&=\nabla\cdot \left( u\nabla (u+v) \right)\\
	\partial_tv&=\nabla\cdot \left( v\nabla (u+v) \right)
    \end{split}
    \textrm{ for }(x,y) , t\in (0,T^*], 
    \label{eq:two_dim_system}
\end{align}
with boundary conditions 
\begin{align}
    \begin{split}
    n\cdot u\nabla (u+v)=0\\
     n\cdot v\nabla (u+v)=0
    \end{split}
    \textrm{ on } \partial\Omega, t\in (0,T^*].
    \label{eq:boundary_condition}
\end{align}

These boundary conditions ensure the conservation of mass in the system, which is also preserved by the suitable finite-volume scheme.
To simulate this system, we use finite-volume discretisation on an equidistant grid where the fluxes adhere to a flux limiter.
To approximate $u$ and $v$, we use a first-order upwind discretisation.
Fluxes are approximated with a second-order central-difference approximation.

The semi-discrete system of ODE's is non-stiff. We use an explicit integration method to maintain the positivity of the solution and acquire and to put no constraints on the Jacobi matrix. Together with the finite-volume space discretisation, this allows for discontinuous initial data.
We use a four-stage Runge-Kutta integration scheme to perform the time integration. 
\subsection{Multi-particle system}
Recall the multi-particle system formulation \eqref{eq:particle_system}.
We simulate the system in the same domain as \eqref{eq:two_dim_system}.
As a potential function $V_\eps$, we use
\begin{equation}
    V_\eps(r/\eps)=\frac{1}{\eps^2\sqrt{2\pi}}e^{\left( -r/(\eps) \right)^2},
	\label{eq:smooth}
\end{equation}
where $r \in \mathbb{R}$ represents the interparticle distance, $c$ is the interaction range parameter and $\eps>0$, modelling a repulsive effect for $r>0$.
This potential formulation is consistent with the potential function description from \cite{Philipowski2007}. In addition the stochasticity allows for modelling the diffusive behaviour present in the continuum system.

Given a particle configuration at time $t$, we use the Euler-Maruyama method (a stochastic variant of the Euler time-integration method) to compute the configuration in $t+\Delta t$.
The positions in time step $t_k$ are updated with:
\begin{align*}
    \begin{split}
        &\Delta X^{i,\varepsilon}=-\frac{1}{n}\sum\limits_{j=1}^n \left(V'_\varepsilon(X^{j,\varepsilon}_{t_k}-X^{i,\varepsilon}_{t_k})+V'_\varepsilon(Y^{j,\varepsilon}_{t_k}-X^{i,\varepsilon}_{t_k})\right)\,\Delta t+\varepsilon\sqrt{\Delta t}W_i,\\
        &\Delta Y^{i,\varepsilon}=-\frac{1}{n}\sum\limits_{j=1}^n \left(V'_\varepsilon(X^{j,\varepsilon}_{t_k}-Y^{i,\varepsilon}_{t_k})+ V'_\varepsilon(Y^{j,\varepsilon}_{t_k}-Y^{i,\varepsilon}_{t_k})\right)\,\Delta t+\varepsilon\sqrt{\Delta t}W_{n+i}.
    \end{split}
    \label{eq:p_system}
\end{align*}
Here, $W_i$ are samples of a standard normal distribution. This term emerges from the distribution of the standard Wiener process: $W_t \sim \mathcal{N}(0,t)$.
We preserve the conservation of mass by implementing reflective boundaries. With these boundaries, we mimic the zero-flux boundaries in the continuum system.

We compute the density by approximating the empirical measure defined in \eqref{eq: empirical measure}. We smoothen the particle positions with a Gaussian kernel $\Phi_h$. This allows us to compare the multi-particle system to its continuum counterpart.
This empirical measure approximation $\mu_h(t)$ for particles $X_1,\dots,X_N$ is defined as $\mu_h(t)= \sum_i^N\delta_{X_t}\ast\Phi_h$, where $h$ represents the smoothing length of the kernel

Figure~\ref{fig:res11} to~\ref{fig:res33} show the results of the two simulations for various points in time for parameters $N_U = N_V=1000$, $\eps=0.3$ and $c = 0.3$. The interpolation kernel has a smoothing length of 0.15.
The continuum solution $w$ is simulated on an equidistant $20\times20$ grid.
We observe similar diffusive and repulsive behaviours in both simulations.

\section{Discussion}
The simulations point out a qualitative agreement with the analytical results. 
Finding the exact relation between the number of particles $N$ and interaction parameter $\eps$ is challenging.
This is due to how density is measured in the multi-particle system (by a finite-radius approximation of the Dirac distribution) and the hidden scaling restrictions that exist on how $N$ and $1/\eps$ go to infinity.
This is illustrated by the following experiments.

For the related problem in \cite{Philipowski2007} we observe the convergence rate \eqref{eq:ineq_philip}. 
We believe that \eqref{eq:ineq_philip} holds in our context as well. From this we induce the condition that $ 1/\eps, N \rightarrow \infty$ under the condition $N$ grows much faster than $1/\eps$.
\begin{equation}
    E\left[ \sup_{0\leq s\leq t} \left|X_s^{N,i,\eps,\delta} - X_s^{i,\delta} \right|^{2} \right] 
    \leq C_1\eps^{-10}\exp\left( \eps^{-12} \right)\frac{1}{N}+C_2\eps^{4}.
    \label{eq:ineq_philip}
\end{equation}
Our numerical experiments indicate that if $N$ and $1/\eps$ increase such that this condition is not respected, the time for the system to reach an equilibrium grows to infinity.

We analyse the density after final time $T*$ for a varying set of parameters. 
We define the residual norm $r_i$ of the particle system as a discrete variant of \eqref{eq:ineq_philip} by performing a sequence of $n$ simulations to find observed densities $\left\{ \mu_h^i(T*) \right\}_{i=1}^n$ and measuring the $\mathcal{L}^2$-norm distance between the densities of simulation $i-1$ and $i$,
\begin{equation}
    r_i = \left|\left| \mu_h^i(T*)  -\mu_h^{i-1}(T*) \right|\right|_{\mathcal{L}^2}.
    \label{eq:residual}
\end{equation}

Figure~\ref{fig:n_fixed_crit} depicts the residual defined in \eqref{eq:residual} for a sequence of simulations with $N=8192$ fixed and $\eps\to0$.
This figure illustrates the transition from systems that converge towards an equilibrium (for $\eps < 2^{-9}$) to stationary systems (for $\eps >  2^{-9}$).

For $\eps = \left( 1/N \right)^\alpha$ and small $\alpha$ we observe the convergence in density profiles. 

\begin{figure}[h]
  \centering
  \begin{minipage}{.45\textwidth}
    \centering
    \includegraphics[width=\textwidth]{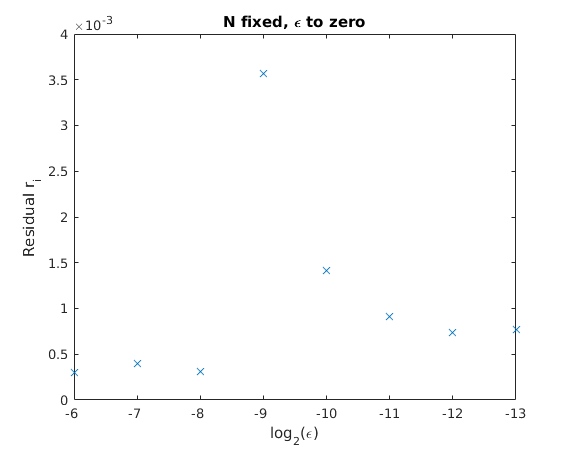}
    \caption{$\mathcal{L}^2$-norm of density residual of particle system after $t=T$ for subsequent simulations, for fixed $N$ and $\eps\to0$.}
    \label{fig:n_fixed_crit}
  \end{minipage}%
  \hfill
  \begin{minipage}{.45\textwidth}
    \centering
    \includegraphics[width=\textwidth]{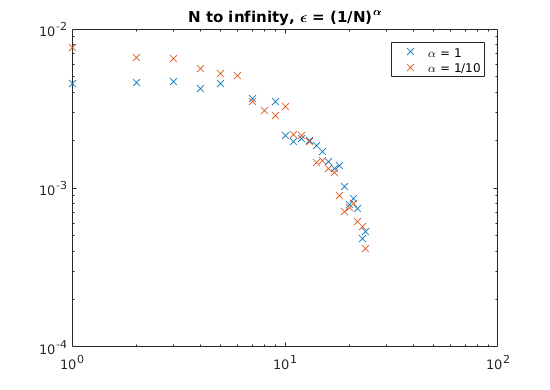}
    \caption{$\mathcal{L}^2$-norm of density residual of the particle system after $t=T$ for subsequent simulations, for two values of $\alpha$.}
    \label{fig:alpha_fixed}
  \end{minipage}
\end{figure}

Finally, the size of the smoothing length $h$ also plays a significant role in representing the interpolated density. 
The finite range of the Dirac interpolation implies that  some mass is lost at the boundaries of the domain.
This effect is visible  when comparing the density profiles at the boundaries of the domain.
Otherwise, a larger smoothing length increases the convergence rate and decreases the distance to the macroscopic density profile.

Further research is required to find an appropriate measure in which experiments converge to the expected macroscopic limit  inside $\Omega$ as well as a correct relation between $N$ and $\eps$.



%
\begin{figure}[h!]
	\begin{minipage}[]{0.3\textwidth}
		\centering
		\includegraphics[width=\textwidth]{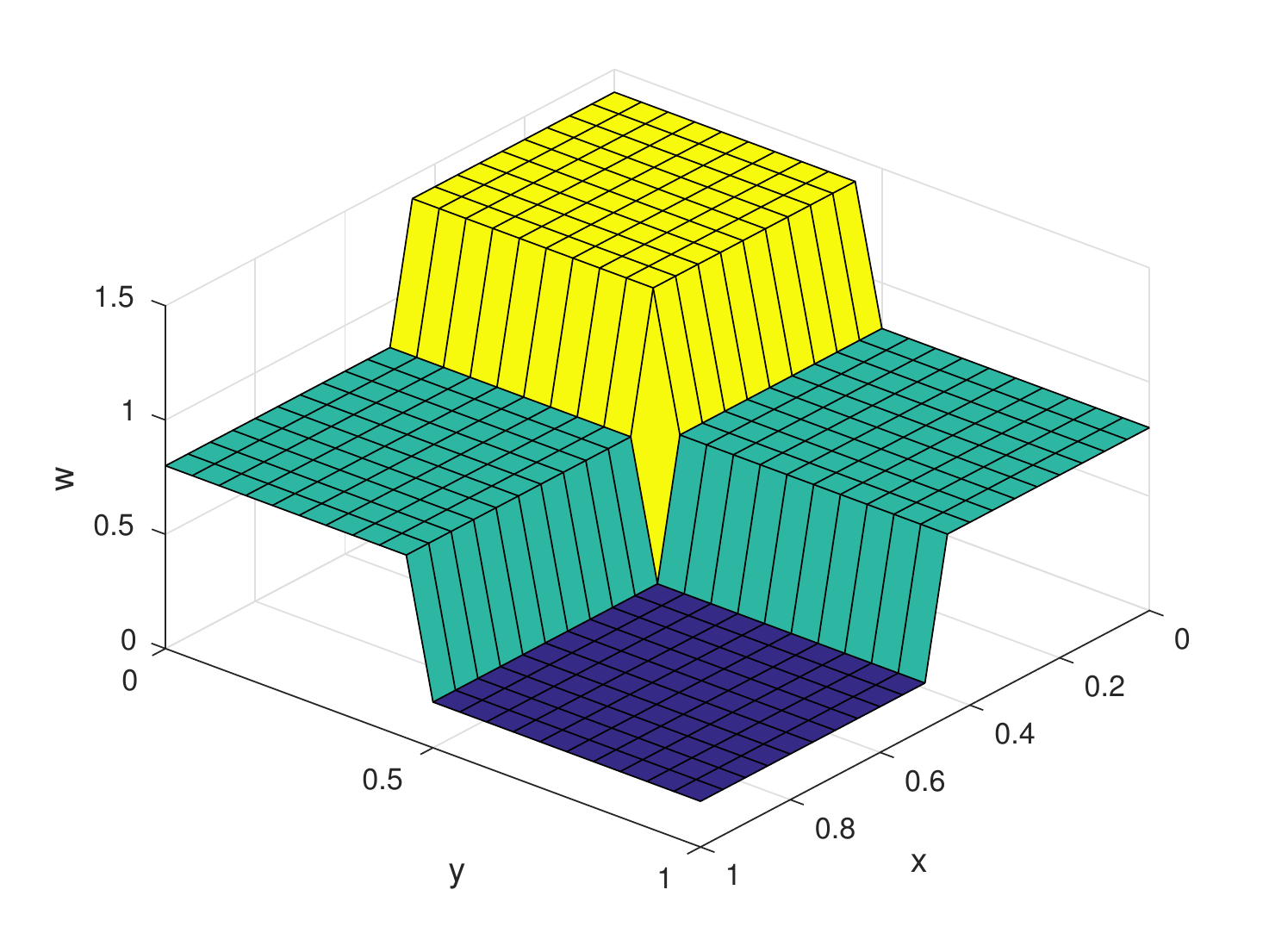}
		\caption{$w(x,y,0)$}
        \label{fig:res11}
	\end{minipage}%
	\hfill
	\begin{minipage}[]{0.3\textwidth}
		\centering
		\includegraphics[width=\textwidth]{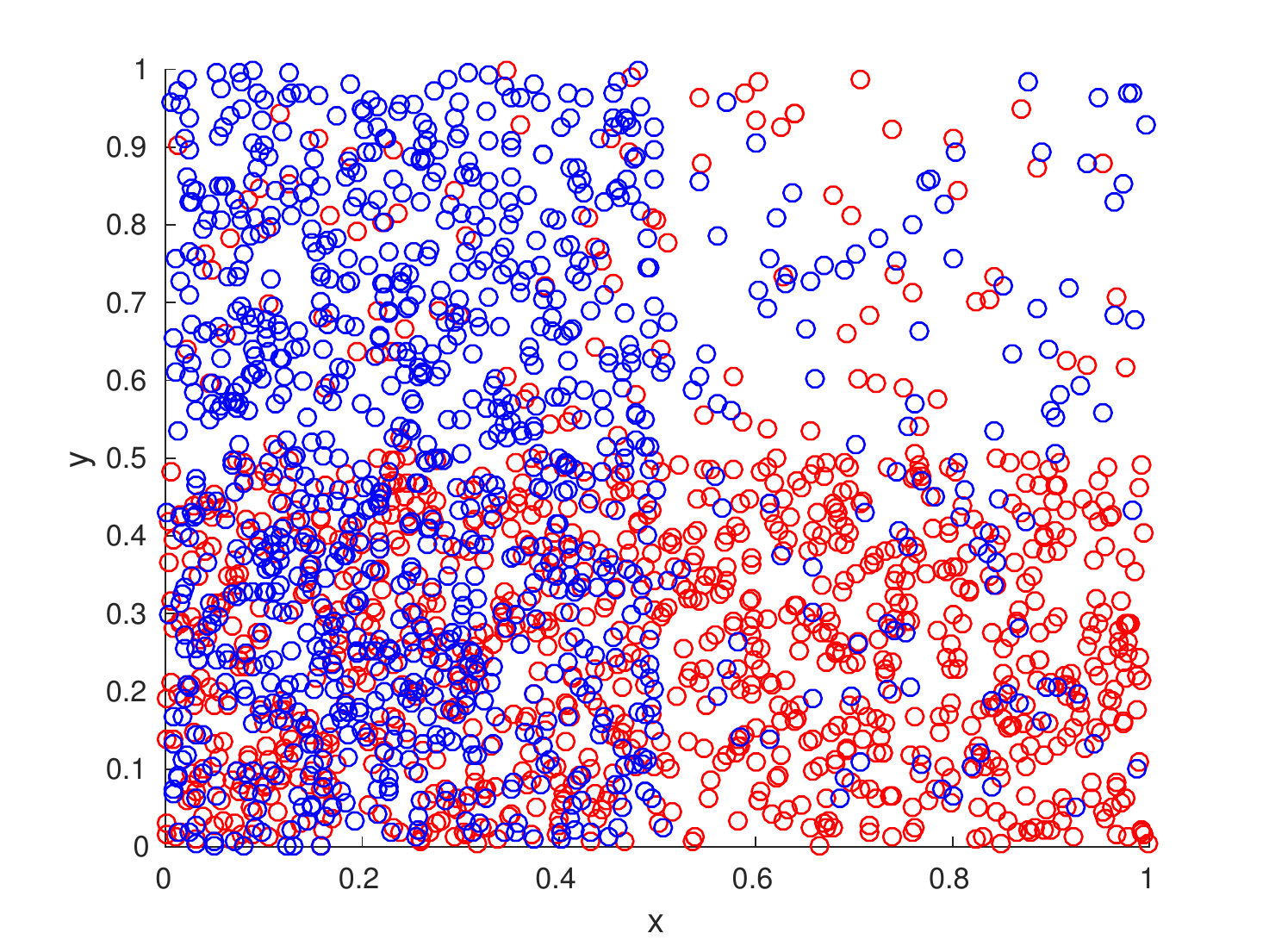}
		\caption{{$X_{0}^{i,\eps}$ (red) and $Y_{0}^{i,\eps}$ (blue)}}
        \label{fig:res12}
	\end{minipage}%
	\hfill
	\begin{minipage}[]{0.3\textwidth}
		\centering
		\includegraphics[width=\textwidth]{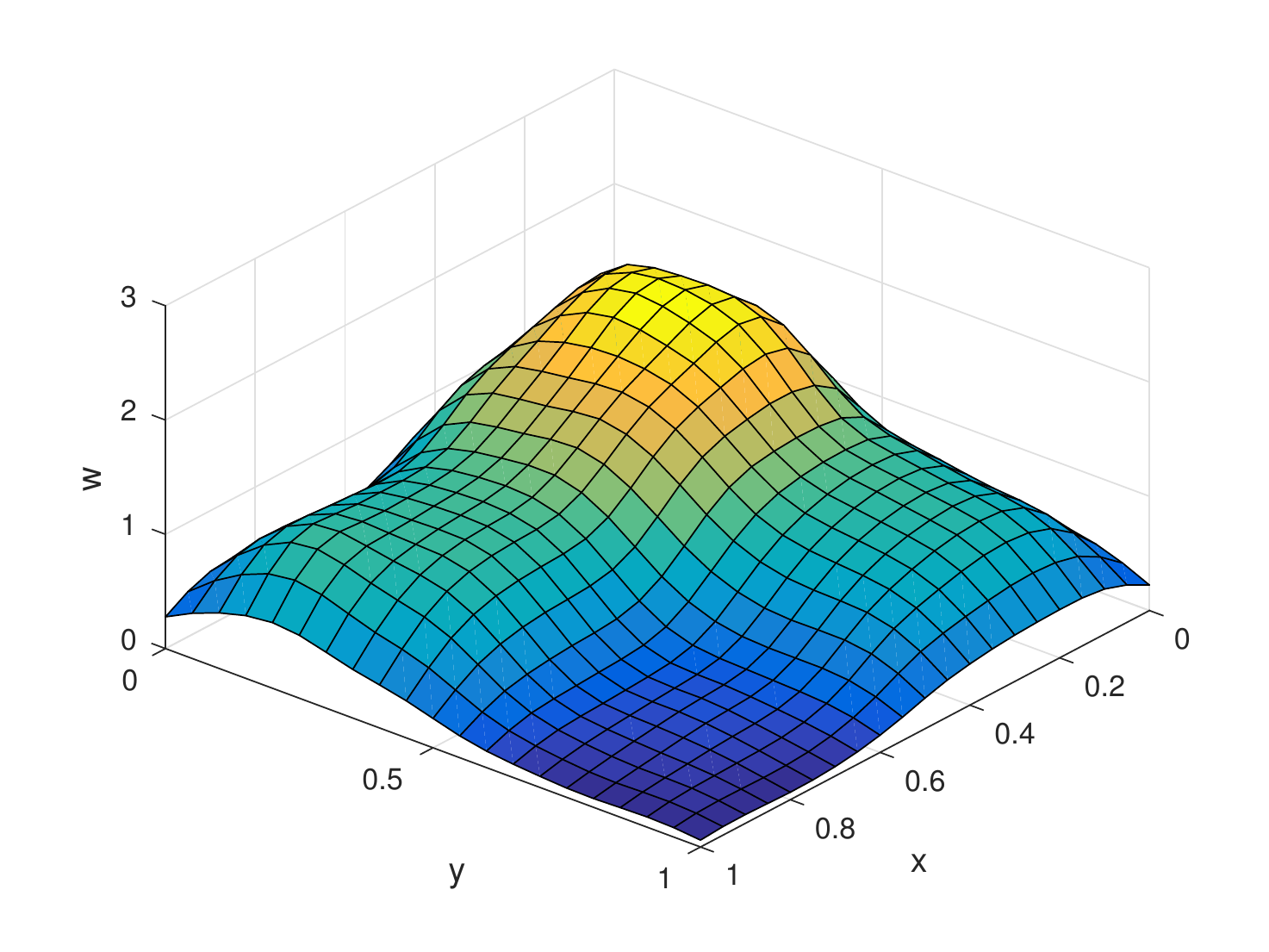}
        \caption{$\mu_{0.15}(0)$ for $X$ and $Y$}
        \label{fig:res13}
	\end{minipage}
    \hfill 
	\begin{minipage}[]{0.3\textwidth}
		\centering
		\includegraphics[width=\textwidth]{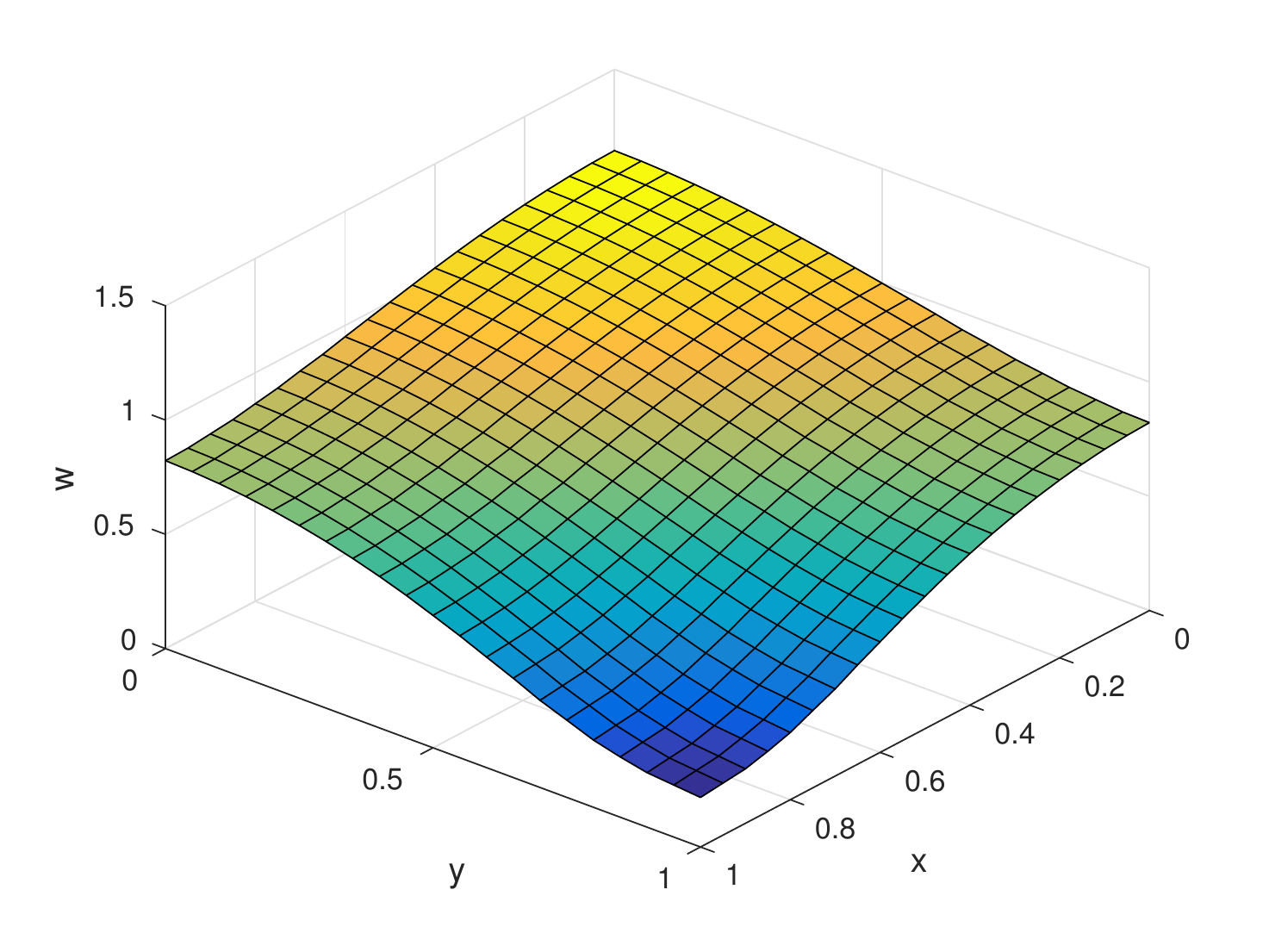}
		\caption{$w(x,y,0.1)$}
        \label{fig:res21}
	\end{minipage}%
	\hfill
	\begin{minipage}[]{0.3\textwidth}
		\centering
		\includegraphics[width=\textwidth]{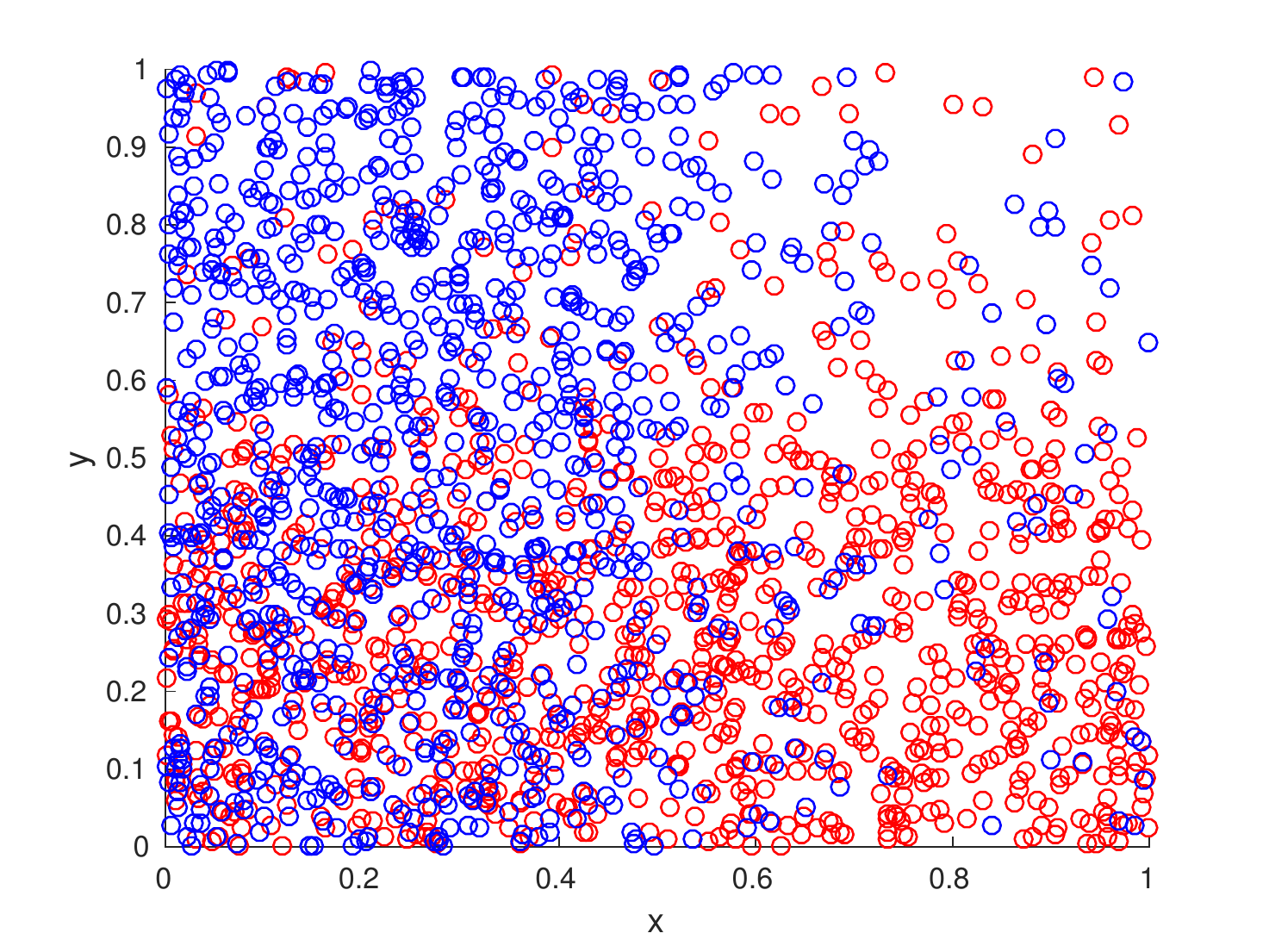}
		\caption{{$X_{0.1}^{i,\eps}$ (red) and $Y_{0.1}^{i,\eps}$ (blue)}}
        \label{fig:res22}
	\end{minipage}%
	\hfill
	\begin{minipage}[]{0.3\textwidth}
		\centering
		\includegraphics[width=\textwidth]{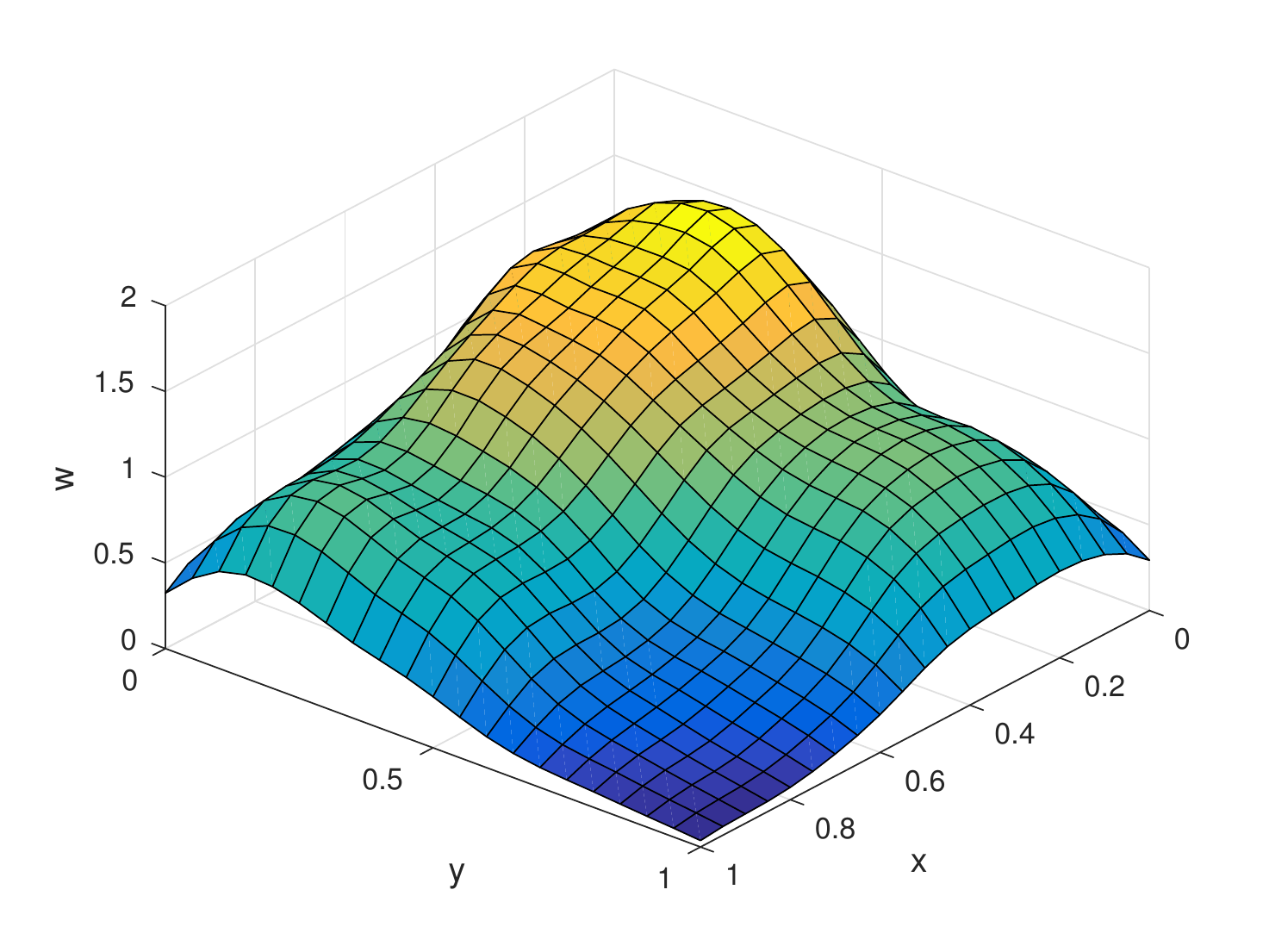}
		\caption{$\mu_{0.15}(0.1)$ for $X$ and $Y$}
        \label{fig:res23}
	\end{minipage}
    \hfill 
	\begin{minipage}[]{0.3\textwidth}
		\centering
		\includegraphics[width=\textwidth]{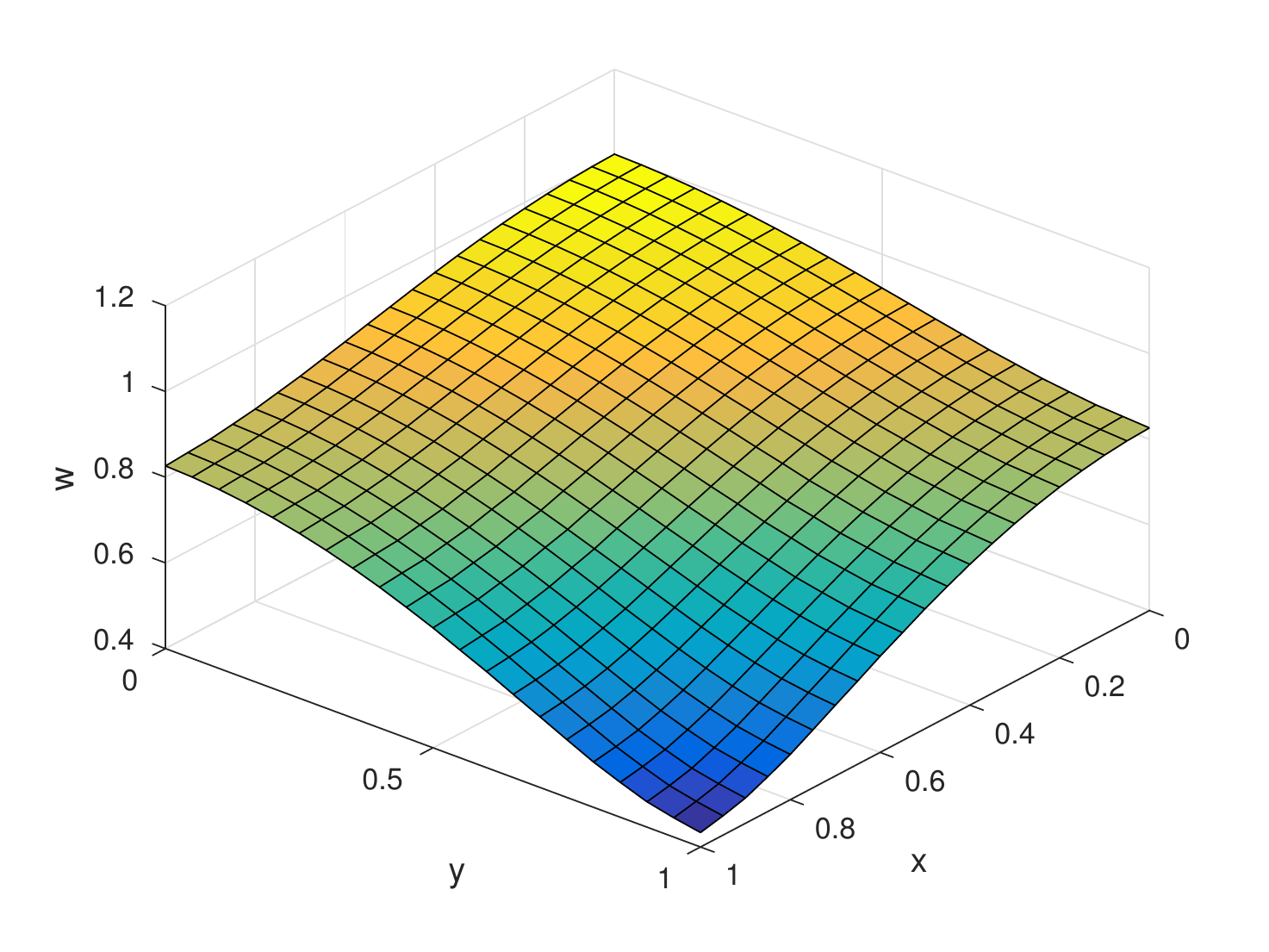}
		\caption{$w(x,y,0.2)$}
        \label{fig:res31}
	\end{minipage}%
	\hfill
	\begin{minipage}[]{0.3\textwidth}
		\centering
		\includegraphics[width=\textwidth]{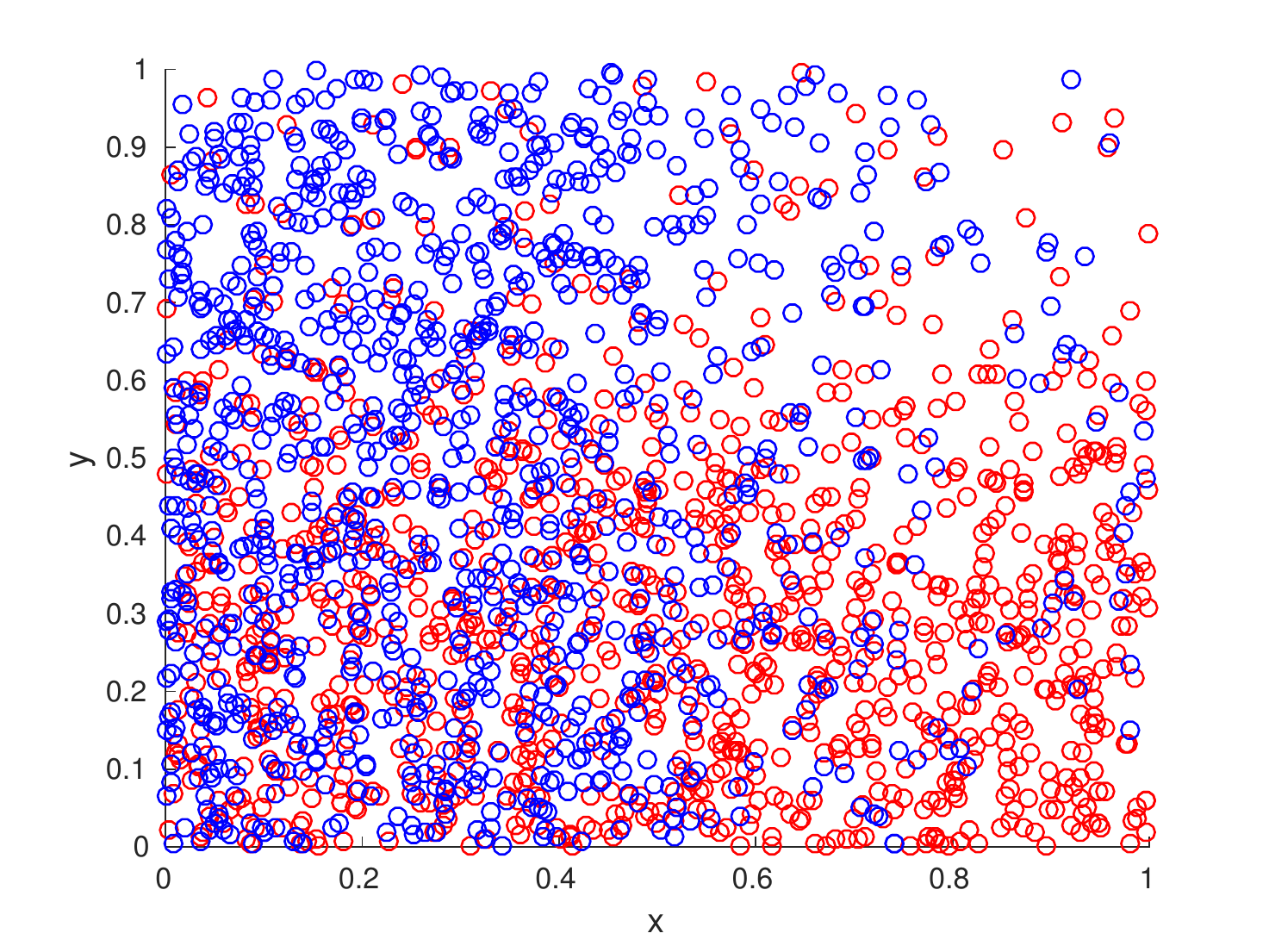}
        \caption{$X_{0.2}^{i,\eps}$ (red) and $Y_{0.2}^{i,\eps}$ (blue)}
        \label{fig:res32}
	\end{minipage}%
	\hfill
	\begin{minipage}[]{0.3\textwidth}
		\centering
		\includegraphics[width=\textwidth]{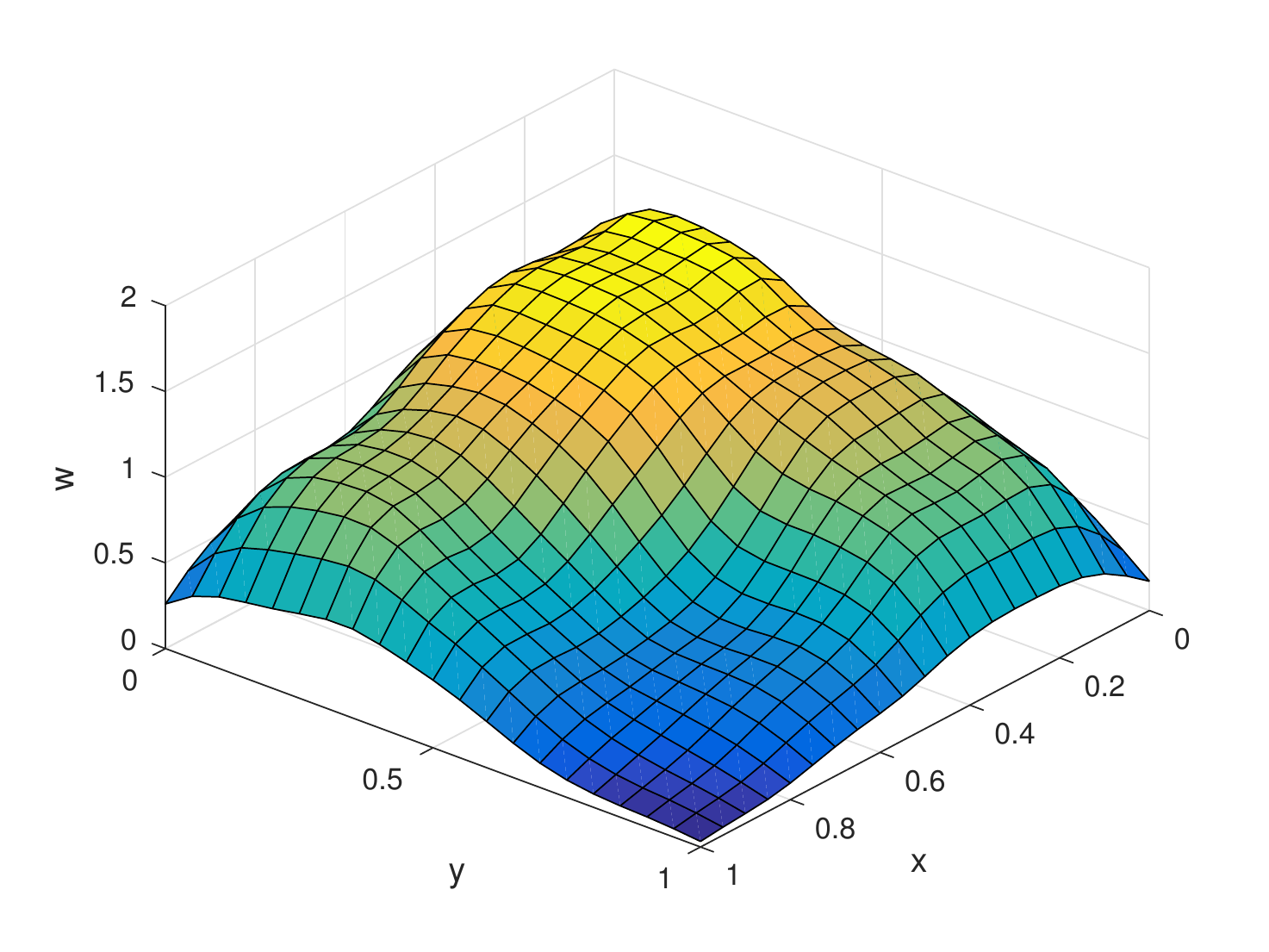}
		\caption{$\mu_{0.15}(0.2)$ for $X$ and $Y$}
        \label{fig:res33}
	\end{minipage}
\end{figure}
\section*{Acknowledgements}
M. H. Duong was  supported by ERC Starting Grant 335120. We wish to thank the referees for useful suggestions.

\end{document}